\documentclass[10pt,leqno]{amsart}
\usepackage{graphicx}
\usepackage{indentfirst,csquotes}
\usepackage{amsmath}
\allowdisplaybreaks[4]
\usepackage{amsthm,amssymb,amsfonts}
\usepackage{graphicx}
\usepackage{indentfirst}
\usepackage{color}
\newtheorem{lemma}{Lemma}[section]
\newtheorem{remark}{Remark}[section]
\newtheorem{theorem}{Theorem}[section]
\newtheorem{definition}{Definition}[section]

\usepackage[colorlinks,linkcolor=blue,citecolor=red,anchorcolor=red]{hyperref}
\author{Yong Lin$^{0}$}
\address{Yong Lin: YMSC, Tsinghua University, Beijing 100084, China}
\email{yonglin@tsinghua.edu.cn}

\author{Shi Wan}
\address{Shi Wan: YMSC, Tsinghua University, Beijing 100084, China}
\email{wans21@mails.tsinghua.edu.cn}
\begin{document}
	\title{Connection Heat Kernel on Connection Lattices and Connection Discrete Torus}
	\date{}
	\begin{abstract}
		By the connection graph we mean an underlying weighted graph with a connection which associates edge set with an orthogonal group. This paper centers its investigation on the connection heat kernels on connection lattices and connection discrete torus. For one dimensional connection lattice, we derive the connection heat kernel expression by doing the Taylor expansion on the exponential function involving normalized connection Laplacian. We introduce a novel connection called product connection and prove that the connection heat kernel on arbitrary high dimensional lattice with product connection equals the Kronecker sum of one dimensional connection lattices' connection heat kernels. Furthermore, if the connection graph is consistent, we substantiate the interrelation between its connection heat kernel and its underlying graph's heat kernel. We define a connection called quotient connection such that discrete torus with quotient connection can be considered as a quotient graph of connection lattice, whose connection heat kernel is demonstrated to be the sum of connection lattices' connection heat kernels. In addition, we derive an alternative expression of connection heat kernel on discrete torus whenever its quotient connection is a constant connection, yielding an equation as a connection graph's trace formula.  	 
	\end{abstract}
		\footnotetext[0]{supported by the National Science Foundation of China(Grant No.12071245)}
	\maketitle

	\section{Introduction}
	
	The heat kernel is the fundamental solution to the heat equation, a partial differential equation that describes how the distribution of heat in a given medium changes over time. It characterizes the behavior of heat propagation, indicating how heat distributes itself across the space as time progresses. As an analytical instrument, heat kernel is useful in delineating certain function spaces\cite{butzer2013semi}, the estimation of whose bounds holds significance across various domains\cite{duong2005duality,jiang2016heat,morris2005evolving}.  On a Riemannian manifold, the heat kernel is typically defined as an exponential function involving the Laplace-Beltrami operator. This operator captures the intrinsic curvature and geometry of the manifold\cite{grigor2006heat,jones2008manifold}. 
	
	The notion of heat kernel can be naturally introduced into graphs with or without connection. It can be defined as either the solution of the (connection) heat equation on graphs or an exponential function involving the graph's (connection) Laplacian operator. For lattices without connection of arbitrary dimensions, F.R.K. Chung et al. not only provided comprehensive elucidations regarding the formulations of the heat kernel and the estimation of its bounds but also derived several hypergeometric equalities utilizing the heat kernel's trace formula in \cite{chung1997combinatorial}. For discrete torus without connection, Alexander Grigor’yan et al. derived the expression of its heat kernel and established equalities concerning trigonometric sums in \cite{grigor2022discrete} as well as Gautam Chinta et al. proved the asymptotic behavior of some spectral invariants through studying the degenerating families of discrete torus in \cite{chinta2010zeta}.
	
	Section 2 includes the mathematical preliminaries that may be needed in later sections. In Section 3 we derive the expressions of connection heat kernel on connection lattices $(\mathbb{Z}^n,\hat{\sigma})$ for all $n\ge 0$, where $\hat{\sigma}$ is a novel connection called product connection. We investigate the property of connection heat kernel on any consistent graph in section 4. Section 5 and 6 involve the derivation for the expression of connection heat kernel on connection discrete torus. We end with introducing two applications of connection heat kernel in section 7.
	
	The main results of our paper are as follows:
	\begin{itemize}
		\item The expression of connection heat kernel on $(\mathbb{Z}^n,\hat{\sigma})$:
		\begin{equation*}
			\begin{aligned}
				&H_t^{\mathbb{Z}^n,\hat{\sigma}}\left(\left(x_1,x_2,\cdots,x_n\right),\left(x_1+a_1,x_2+a_2,\cdots,x_n+a_n\right)\right)\\
				&=\prod_{i=1}^{n}\left((-1)^{|a_i|} \sum_{k\ge 0}\frac{C_{2k}^{k+|a_i|}}{k!} (-\frac{t}{2n})^k\right) \sigma^{(1)}_{P_{x_1\rightarrow x_1+a_1}}\otimes \cdots\otimes \sigma^{(n)}_{P_{x_n\rightarrow x_n+a_n}}
			\end{aligned}
		\end{equation*}
		\item The correlation between the connection heat kernel $H_t^{\sigma}$ on any consistent graph $(\Gamma,\sigma)$ and the heat kernel $H_t$ on its underlying graph $\Gamma$:
		\[ H_t^{\sigma}(x,y)=H_t(x,y)\sigma_{P_{x\rightarrow y}} \] 
		\item The expression of connection heat kernel on  $(\mathbb{Z}^n/M\mathbb{Z}^n,\hat{\sigma}^{Q_{M\mathbb{Z}^n}})$ is:
		\[ 
		\begin{aligned}
			&H_t^{\hat{\sigma}^{Q_{M\mathbb{Z}^n}}}([x],[y])\\
			&=\sum_{a\in M\mathbb{Z}^n}\prod_{i=1}^{n}\left((-1)^{|y_i+a_i-x_i|} \sum_{k\ge 0}\frac{C_{2k}^{k+|y_i+a_i-x_i|}}{k!} (-\frac{t}{2n})^k\right)\\ &\sigma^{(1)}_{P_{x_1\rightarrow y_1+a_1}}\otimes \sigma^{(2)}_{P_{x_2\rightarrow y_2+a_2}}\otimes\cdots\otimes \sigma^{(n)}_{P_{x_n\rightarrow y_n+a_n}}
		\end{aligned}
		\]
		\item A matrix equation based on connection heat kernel on $(\mathbb{Z}^n/M\mathbb{Z}^n,\hat{\sigma}^{Q_{M\mathbb{Z}^n}})$:
		\begin{tiny}
			\begin{equation*}
				\begin{aligned}
					&\sum_{a\in M\mathbb{Z}^n}\prod_{i=1}^{n}\left((-1)^{|y_i+a_i-x_i|} \sum_{k\ge 0}\frac{C_{2k}^{k+|y_i+a_i-x_i|}}{k!} (-\frac{t}{2n})^k\right)
					\sigma^{y_1+a_1-x_1}_1\otimes\cdots\otimes \sigma^{y_n+a_n-x_n}_n\\
					&=\frac{1}{det M}\sum_{w\in(M^*)^{-1}\mathbb{Z}^n /\mathbb{Z}^n}e^{-t}e^{2\pi i\langle w,x-y\rangle}e^{\frac{t}{n}cos(2\pi w_1)\sigma_1\oplus\frac{t}{n}cos(2\pi w_2)\sigma_2\oplus\cdots\oplus\frac{t}{n}cos(2\pi w_n)\sigma_n }
				\end{aligned}
			\end{equation*} 
		\end{tiny}
		
	\end{itemize}
	
	\section{Preliminary}
	\subsection{Graph's Standard Laplacian and Connection Laplacian}
	Suppose $\Gamma=(V,E,w)$ is an undirected weighted graph where $V$ is the vertex set, $E$ is the edge set and all edge weight $w_{uv}=w_{vu}>0$ if and only if $(u,v)\in E$. The degree of a vertex $v $ is defined as $d(v):=\sum_{u\sim v}w_{vu}$. The degree matrix $D$ of $\Gamma$ is a diagonal matrix consisting of the degrees of all vertices. The adjacency matrix $A$ of $\Gamma$ is defined by :
	
	$$A(u,v)=\left\{ 	\begin{aligned}
		w_{uv}  \quad &if (u,v)\in E \\
		0 \quad&else 
	\end{aligned}\right.$$
	The standard Laplacian of $\Gamma$ is defined as $L:=D-A$ and the normalized standard Laplacian is defined as $\mathcal{L}:=I-D^{-\frac{1}{2}}AD^{-\frac{1}{2}}=D^{-\frac{1}{2}}LD^{-\frac{1}{2}}$. Let $C(\Gamma,\mathbb{R})$ denote the space of functions $f:V(\Gamma)\rightarrow \mathbb{R}$. The standard Laplacian is an operator on $C(\Gamma,\mathbb{R})$ and the action is 
	\[ Lf(u)=\sum_{v\sim u}w_{uv}(f(u)-f(v)) \]
	
	The connection of $\Gamma$ is a map from the set of all directed edges to orthogonal linear transformations, which assigns an orthogonal matrix $\sigma_{uv}$ to every directed edge $(u,v)$ satisfying $\sigma_{uv}=(\sigma_{vu})^{-1}$. That is, $\sigma:E\rightarrow O(d)$ satisfies $\sigma_{uv}\sigma_{vu}=I_{d\times d}$.
	We call $(\Gamma,\sigma)=(V,E,w,\sigma)$ connection graph which has $\Gamma=(V,E,w)$ as underlying graph. If orthogonal transformation acts on d-dimensional space, we say connection $\sigma$ is d-dimensional. The connection degree matrix of $(\Gamma,\sigma) $ is a block-diagonal matrix $D^{\sigma}$ with diagonal block $D^{\sigma}(u,u)=d(u)I_{d\times d}$. The connection adjacency matrix of $(\Gamma,\sigma)$  is defined by:
	$$A^{\sigma}(u,v)=\left\{ 	\begin{aligned}
		w_{uv}\sigma_{uv}  \quad &if (u,v)\in E \\
		0_{d\times d} \quad&else 
	\end{aligned}\right.$$
	The connection Laplacian of $(\Gamma,\sigma)$ is defined as $L^{\sigma}:=D^{\sigma}-A^{\sigma}$ and the normalized connection Laplacian $\mathcal{L}^{\sigma}$ is defined as $\mathcal{L}^{\sigma}:=I-(D^{\sigma})^{-\frac{1}{2}}A^{\sigma}(D^{\sigma})^{-\frac{1}{2}}=(D^{\sigma})^{-\frac{1}{2}}L(D^{\sigma})^{-\frac{1}{2}}$.
	Let $C((\Gamma,\sigma),\mathbb{R}^d)$ denote the space of functions $f:V(\Gamma)\rightarrow \mathbb{R}^d$. The connection Laplacian is an operator on $C((\Gamma,\sigma),\mathbb{R}^d)$ and the action is 
	\[ L^{\sigma}f(u)=\sum_{v\sim u}w_{uv}(f(u)-\sigma_{uv}f(v)) \]
	According to Courant-Fischer Theorem, we can study the eigenvalues of normalized connection Laplacian $\mathcal{L}^{\sigma}$ by examining its Rayleigh quotient
	\[ \mathcal{R}(f)=\frac{f^T\mathcal{L}^{\sigma}f}{f^Tf} \]
	where $f\in C((\Gamma,\sigma),\mathbb{R}^d)$ is regarded as a vector in $\mathbb{R}^{d|V|}$.
	Let $g=(D^{\sigma})^{-\frac{1}{2}}f$, then
	\begin{align*}
		\mathcal{R}(f)&=\frac{g^TL^{\sigma}g}{g^TD^{\sigma}g}\\
		&=\frac{\sum_{u\sim v}w_{uv}||g(u)-\sigma_{uv}g(v)||^2}{\sum_{v\in V} d_v||g(v)||^2}\\
		&\le \frac{2\sum_{u\sim v}w_{uv}\left(||g(u)||^2+||\sigma_{uv}g(v)||^2\right)}{\sum_{v\in V} d_v||g(v)||^2}\\
		&=\frac{\sum_{u,v}w_{uv}||g(u)||^2+\sum_{u, v}w_{uv}||g(v)||^2}{\sum_{v\in V} d_v||g(v)||^2}\\
		&=\frac{\sum_{u}d_{u}||g(u)||^2+\sum_{v}d_{v}||g(v)||^2}{\sum_{v\in V} d_v||g(v)||^2}\\
		&=2
	\end{align*}
	Therefore, all eigenvalues of $\mathcal{L}^{\sigma}$ are contained in $[0,2]$ and $\mathcal{L}^{\sigma}$ is a bounded operator on $C((\Gamma,\sigma),\mathbb{R}^d) $.
	\subsection{Graph's Heat Kernel and Connection Heat Kernel}
	Consider the heat equation on graph $\Gamma$ without connection:
	\[ \left\{
	\begin{aligned}
		(\frac{\partial}{\partial t}+\mathcal{L})f(t,x)&=0\\
		f(0,x)&=\delta_y(x)
	\end{aligned}
	\right.
	\]
	where $\delta_y$ is the characteristic function for the vertex $y$. The heat kernel of $\Gamma$ is the solution of the above equation, denoted by $k(t,x)$. For $\forall f\in C(\Gamma,\mathbb{R})$, $(\frac{\partial}{\partial t}+\mathcal{L})e^{-t\mathcal{L}}f=-\mathcal{L}e^{-t\mathcal{L}}f+\mathcal{L}e^{-t\mathcal{L}}f=0 $. Then $k(t,x)=e^{-t\mathcal{L}}\delta_y(x)$.
	
	Assume $\Gamma$ is finite and the number of vertices is $n$. It's known that standard Laplacian and normalized standard Laplacian of $\Gamma$ is symmetric positive semi-definite with real eigenvalues. Suppose $0=\lambda_0\le \lambda_1\le \cdots \le \lambda_{n-1}$ are the eigenvalues of $\mathcal{L}$ and corresponding orthonormal eigenfunctions are $\left\{\phi_i\right\}_{i=0}^{n-1}$.  
	\[ \delta_{y}(x)=\sum_{i=0}^{n-1}\langle \phi_i,\delta_y\rangle \phi_i(x)=\sum_{i=0}^{n-1}\overline{\phi_i(y)}\phi_i(x) \]
	Then the characteristic representation of heat kernel is
	\[ k(t,x)=e^{-t\mathcal{L}}\delta_{y}(x)=\sum_{i=0}^{n-1} e^{-t\mathcal{L}}\phi_i(x)\overline{\phi_i(y)}=\sum_{i=0}^{n-1} e^{-t\lambda_i}\phi_i(x)\overline{\phi_i(y)}\]
	We denote the operator $e^{-t\mathcal{L}}$ by $H_t$. Then the heat kernel  $k(t,x)=H_t\delta_{y}(x)=H_t(x,y)$. Therefore, we only need to specify the entry of $H_t$. In this paper, we refer to $H_t$ as heat kernel.
	
	The connection heat kernel of connection graph $(\Gamma,\sigma)$ can be defined analogously. We denote the operator $e^{-t\mathcal{L}^{\sigma}}$ by $H_t^{\sigma}$ and call $H_t^{\sigma} $ connection heat kernel of $(\Gamma,\sigma)$. Note that $H_t^{\sigma}(x,y)$ is $d\times d$ matrix if $\sigma$ is d-dimensional. Assume the underlying graph $\Gamma$ is finite and $|V|=n$. For $g \in C((\Gamma,\sigma),\mathbb{R}^d)$, we have $g^TL^{\sigma}g=\sum_{u\sim v}w_{uv}||g(u)-\sigma_{uv}g(v)||^2\ge 0$, implying
	the connection Laplacian $L^{\sigma}$ of $(\Gamma,\sigma)$ is positive semi-definite. Therefore, $L^{\sigma}$ is a real symmetric positive semi-definite matrix. Since the normalized connection Laplacian $\mathcal{L}^{\sigma}$ is similar to $L^{\sigma}$, $\mathcal{L}^{\sigma}$ is also a real symmetric positive semi-definite matrix and its eigenvalues are non-negative. Suppose $0\le \mu_0\le \mu_1\le \cdots \le \mu_{nd-1}$ are the eigenvalues of $\mathcal{L}^{\sigma}$ and corresponding orthonormal eigenfunctions are $\left\{\Phi_i\right\}_{i=0}^{nd-1}$. Then  \[ H_t^{\sigma}(x,y)=\sum_{i=0}^{nd-1} e^{-t\mu_i}\Phi_i(x)\overline{\Phi_i(y)}^T \]
	
	For $x,y \in V$, graph heat kernel $H_t(x,y)$ is a number while connection heat kernel $H_t^{\sigma}(x,y)$ is a matrix. For connection graph, it's more meaningful to pay attention to each block entry rather than every entry when studying the connection heat kernel.

	\section{The Connection Heat Kernel On Connection Lattices}
	First, we derive the expression of connection heat kernel on one dimensional connection lattice $\mathbb{Z}$ with a connection $\sigma:E(\mathbb{Z})\rightarrow O(d)$, where $d$ is an positive integer. Assume every edge weight is equal to $1$, that is $w_{xy}=w_{yx}=1,\forall x\sim y$. Then the normalized connection Laplacian $\mathcal{L}^{\sigma}$ of $(\mathbb{Z},\sigma)$ has the following block-triangular form:
	\begin{equation*}
		\begin{pmatrix}
			\ddots&\ddots& \ddots	&  & & & \\
			&	-\frac{1}{2} \sigma_{x,x-1} & I_{d\times d} & -\frac{1}{2} \sigma_{x,x+1} & & &  \\
			&	& -\frac{1}{2} \sigma_{x+1,x} & I_{d\times d} & -\frac{1}{2} \sigma_{x+1,x+2}& & \\
			& &	& \ddots & \ddots & \ddots&   \\
		\end{pmatrix}
	\end{equation*}
	Note that $\mathcal{L}^{\sigma}$ is infinite dimensional and the multiplication of infinite dimensional matrices may appear. Luckily, in later computation we only need to deal with the powers of infinite dimensional matrices whose rows and columns have finite non-zero entries, avoiding the bad case where associativity and distributivity may not hold for infinite-dimensional matrices.
	
	We write $\mathcal{L}^{\sigma}$ as $\hat{I}-\frac{1}{2}U^{\sigma}-\frac{1}{2}(U^{\sigma})^{-1}$, where $U^{\sigma}$ is the upper triangular block part of $\mathcal{L}^{\sigma}$, $\hat{I}$ is the diagonal block part of $\mathcal{L}^{\sigma}$ and $(U^{\sigma})^{-1}$ is the lower triangular block part of $\mathcal{L}^{\sigma}$.

		\[
		\hat{I}=	\begin{pmatrix}
			&\ddots & & & &\\
			&  & I_{d\times d} & & & \\
			&	&  & I_{d\times d} & &  \\
			& &	&  & \ddots &  \\	
		\end{pmatrix}
		\]
	
			\[ 
			U^{\sigma}=\begin{pmatrix}
				&\ddots  & \ddots & & & \\
				&	& 0_{d\times d} & \sigma_{_{x\text{,}x+1}} & &  \\
				& &	& 0_{d\times d} & \sigma_{_{x+1\text{,}x+2}} &  \\	
				& &  &	&\ddots&\ddots
			\end{pmatrix}
			\]	
		
			\[ 
			(U^{\sigma})^{-1}=\begin{pmatrix}
				\ddots	& &  & & & \\
				\sigma_{_{x+1\text{,}x}}& 0_{d\times d} &  & & &  \\
				&\sigma_{_{x+2\text{,}x+1}}	&0_{d\times d}  &  & &  \\
				&	& \ddots & \ddots & & \\	
			\end{pmatrix}
			\]

	\begin{theorem}\label{Zexpre}
		In $(\mathbb{Z},\sigma)$, $\forall x \in \mathbb{Z}, a\in \mathbb{Z}^{+} $: 
		\begin{itemize}
			\item [(1)]The diagonal block of connection heat kernel of $(\mathbb{Z},\sigma)$ is :
			\[ H_t^{\sigma}(x,x)=\sum_{k\ge 0} \frac{C_{2k}^{k}}{k!}(\frac{-t}{2})^k I_{d\times d} \]
			\item [(2)]The off-diagonal block of connection heat kernel of $(\mathbb{Z},\sigma)$ is :
			\begin{equation*}
				\begin{aligned}
					&H^{\sigma}_t(x,x+a)\\
					&=(-1)^a \sum_{k\ge 0}\frac{C_{2k}^{k+a}}{k!} (-\frac{t}{2})^k\sigma_{x,x+1}\sigma_{x+1,x+2}\cdots\sigma_{x+a-1,x+a}\\
					&H^{\sigma}_t(x,x-a)\\
					&=(-1)^a \sum_{k\ge 0}\frac{C_{2k}^{k+a}}{k!} (-\frac{t}{2})^k\sigma_{x,x-1}\sigma_{x-1,x-2}\cdots\sigma_{x-(a-1),x-a}
				\end{aligned}
			\end{equation*}
		\end{itemize}	
	\end{theorem} 
	
	\begin{proof}
		$$
		\begin{aligned}
			H_t^{\sigma}&=exp(-t\mathcal{L}^{\sigma})\\
			&=\sum_{k\ge 0} \frac{(-t)^k}{k!}(\mathcal{L}^{\sigma})^k\\
			&=I_{d\times d}-t\mathcal{L}^{\sigma}+\frac{t^2}{2}(\mathcal{L}^{\sigma})^2 -\cdots		
		\end{aligned}
		$$
		From the above expansion of $H_t^{\sigma}$, it is sufficient to know the blocks of $ (\mathcal{L}^{\sigma})^k$ for every $k\ge 0$ in order to obtain each block of connection heat kernel. 
		
		Since $(\mathcal{L}^{\sigma})^k=(\hat{I}-\frac{1}{2}U^{\sigma}-\frac{1}{2}(U^{\sigma})^{-1})^k $ and all of the diagonal blocks of $(U^{\sigma})^s,(U^{\sigma})^{-s}$ are  $d\times d$ zero matrices when $s\ge 1$, $I_{d\times d}$ multiplied by the coefficient of $(U^{\sigma})^0$ in $(\hat{I}-\frac{1}{2}U^{\sigma}-\frac{1}{2}(U^{\sigma})^{-1})^k$ is the the diagonal block of $(\mathcal{L}^{\sigma})^k$ and each diagonal block is the same. From $(\hat{I}-\frac{1}{2}U^{\sigma}-\frac{1}{2}(U^{\sigma})^{-1})^k=(-\frac{1}{2})^k(\sqrt{U^{\sigma}}-\sqrt{(U^{\sigma})^{-1}})^{2k} $, we can see the coefficient of $(U^{\sigma})^0$ is $ (\frac{1}{2})^k C_{2k}^k$. So every diagonal block of the $(\mathcal{L}^{\sigma})^k$ is $(\frac{1}{2})^k C_{2k}^k I_{d\times d} $. Consequently, we can know every diagonal block of connection heat kernel is $\sum_{k\ge 0}\frac{C_{2k}^{k}}{k!}(\frac{-t}{2})^k I_{d\times d}$.
		
		Because $U^{\sigma}$ has non-zero blocks only on the superdiagonal, $(U^{\sigma})^s$  has non-zero blocks only on the $\{(x,x+s)\}_{x\in \mathbb{Z}}$ when $s\ge 1$.Since $(U^{\sigma})^{-1}$ is the transpose of $U^{\sigma}$, $(U^{\sigma})^{-s}$  has non-zero blocks only on the $\{(x,x-s)\}_{x\in \mathbb{Z}}$.  Obviously, $$(U^{\sigma})^s(x,x+s)=\sigma_{x,x+1}\sigma_{x+1,x+2}\cdots \sigma_{x+s-1,x+s}$$ 
		$$ (U^{\sigma})^{-s}(x,x-s)=\sigma_{x,x-1} \sigma_{x-1,x-2} \cdots  \sigma_{x-(s-1),x-s}$$ 
		
		Then $ \sigma_{x,x+1}\sigma_{x+1,x+2}\cdots \sigma_{x+s-1,x+s}$ multiplied by the coefficient of $(U^{\sigma})^a$ in $(\hat{I}-\frac{1}{2}U^{\sigma}-\frac{1}{2}(U^{\sigma})^{-1})^k$ is the block matrix on $(x,x+a)$ of $(\mathcal{L}^{\sigma})^k$ as well as $\sigma_{x,x-1} \sigma_{x-1,x-2} \cdots  \sigma_{x-(s-1),x-s} $ multiplied by the coefficient of $(U^{\sigma})^{-a}$ in $(\hat{I}-\frac{1}{2}U^{\sigma}-\frac{1}{2}(U^{\sigma})^{-1})^k$ is the block matrix on $(x,x-a)$ of $(\mathcal{L}^{\sigma})^k$. Both of the coefficients are $(-1)^a C_{2k}^{k+a}(\frac{1}{2})^k$. When $2k<k+a$,we regard $C_{2k}^{k+a}$ as $0$. As a result, we can get the expressions of $H^{\sigma}_t(x,x\pm a) $ by adding up $(\mathcal{L}^{\sigma})^k(x,x\pm a), k\ge 0 $.
	\end{proof}
	
	Next we introduce a type of connection termed product connection on the Cartesian product of connection graphs, which integrates individual connections via the Kronecker product of matrices. 
	
	\begin{definition}
		If $A$ is an $m\times n$ matrix and $B$ is a $p\times q$ matrix, then the Kronecker product $A\otimes B$ is the $mp\times nq$ matrix:
		\[ A\otimes B=
		\begin{pmatrix}
			a_{11}B&\cdots&a_{1n}B\\
			\vdots&\ddots&\vdots\\
			a_{m1}B&\cdots&a_{mn}B\\
		\end{pmatrix}
		\]
		If $A$ is an $m\times m$ matrix and $B$ is a $n\times n$ matrix, then the Kronecker sum $A\oplus B$ is
		\[A\oplus B:=A\otimes I_{n\times n} +I_{m\times m}\otimes B \]
	\end{definition}
	
	\begin{definition}
		Suppose $(\Gamma_{i},\sigma^{(i)}),i=1,2$ are connection graphs with $\sigma^{(i)}:E(\Gamma_{i})\rightarrow O(d_i)$. Let $\Gamma_{1} \square \Gamma_{2} $ be the Cartesian product graph of $\Gamma_{1}$ and $\Gamma_{2}$. Define a connection $\hat{\sigma}:E(\Gamma_{1}\square \Gamma_{2})\rightarrow O(d_1d_2)$ as follows:
		\begin{equation*}\label{eq3}
			\begin{aligned}
				\hat{\sigma}_{(x,y)(x^{'},y)}	&=\sigma^{(1)}_{xx^{'}}\otimes I_{d_2\times d_2}\\
				\hat{\sigma}_{(x,y)(x,y^{'})}	&=I_{d_1\times d_1}\otimes \sigma^{(2)}_{yy^{'}}\\
			\end{aligned}
		\end{equation*} 
		We call $\hat{\sigma}$ product connection on $\Gamma_{1} \square \Gamma_{2} $, denoted by $\sigma^{(1)}\otimes \sigma^{(2)}$.
	\end{definition}
	
	\begin{remark}
		In fact, we have to make it sure that $\hat{\sigma}$ we define is a connection on $\Gamma_{1} \square \Gamma_{2} $, which can be easily checked because the Kronecker product of two orthogonal matrices is still an orthogonal matrix and the inverse of the Kronecker product of two matrices is the Kronecker product of their inverses. 
	\end{remark}
	
	Let $\mathcal{L^{\sigma^{(i)}}}$ and $ \mathcal{L}^{\hat{\sigma}}$ be the normalized connection Laplacian of $(\Gamma_{i},\sigma^{(i)}) $ and the connection Cartesian product $(\Gamma_{1} \square \Gamma_{2},\hat{\sigma})$ respectively. Here we assume  $ \Gamma_i$ is $R_i-$regular and has simple weight. Next we demonstrate that $ \mathcal{L}^{\hat{\sigma}}$ is equal to a certain Kronecker sum involving $\mathcal{L^{\sigma^{(1)}}}$ and $\mathcal{L^{\sigma^{(2)}}}$. 
	
	Firstly, we find one basis of $C((\Gamma_{1} \square \Gamma_{2},\hat{\sigma}),\mathbb{R}^{d_1d_2})$ which is the domain of $ \mathcal{L}^{\hat{\sigma}}$. Let $n_i:=|V(\Gamma_{i})|,i=1,2$ and $n_i$ can be positive infinity if $\Gamma_{i}$ is infinite. Assume $V(\Gamma_{1})=\{x_{i_1}\}_{i_1=1}^{n_1}$ and $V(\Gamma_{2})=\{y_{i_2}\}_{i_2=1}^{n_2}$. 
	
	$\forall 1\le i_1\le n_1,1\le j_1 \le d_1$, define $\delta_{i_1}^{j_1}$ on $\Gamma_{1} $ as follows:
	\[ \delta_{i_1}^{j_1}:V(\Gamma_{1})\rightarrow \mathbb{R}^{d_1} \]
	\begin{equation*}
		\delta_{i_1}^{j_1}(z)=\left\{	\begin{aligned}
			e_{j_1} \quad &z=x_{i_1}\\
			\vec{0}  \quad &else
		\end{aligned}
		\right.
	\end{equation*}
	where $e_{j_1}$ is the unit vector in $\mathbb{R}^{d_1}$ which takes $1$ on the $j_1-$th entry.
	$ \forall 1\le i_2\le n_2,1\le j_2 \le d_2$ we can define $\chi_{i_2}^{j_2}$ on $\Gamma_{2}$ similarly.
	
	$\forall 1\le i_1\le n_1,1\le j_1 \le d_1,1\le i_2\le n_2,1\le j_2 \le d_2 $, we define $\delta_{i_1}^{j_1}\otimes \chi_{i_2}^{j_2}$ as follows:
	\[ \delta_{i_1}^{j_1}\otimes \chi_{i_2}^{j_2}:V(\Gamma_{1}\square\Gamma_{2} )\rightarrow \mathbb{R}^{d_1d_2} \]
	\[ \delta_{i_1}^{j_1}\otimes \chi_{i_2}^{j_2}(x,y)=\delta_{i_1}^{j_1}(x)\otimes \chi_{i_2}^{j_2}(y) \]
	Then $\{\delta_{i_1}^{j_1}\otimes \chi_{i_2}^{j_2}\}_{i_1=1,j_1=1,i_2=1,j_2=1}^{n_1,d_1,n_2,d_2}$ form the basis of $ C(\Gamma_{1}\square \Gamma_{2},\hat{\sigma})$.
	
	\begin{lemma}\label{lemma1}
		$$\mathcal{L}^{\hat{\sigma}} \delta_{i_1}^{j_1}\otimes \chi_{i_2}^{j_2}(x,y)= \frac{R_1}{R_1+R_2}\mathcal{L}^{\sigma^{(1)}}\delta_{i_1}^{j_1}(x)\otimes \chi_{i_2}^{j_2}(y)+\frac{R_2}{R_1+R_2}\delta_{i_1}^{j_1}(x)\otimes \mathcal{L}^{\sigma^{(2)}}\chi_{i_2}^{j_2}(y)$$
	\end{lemma}
	
	\begin{proof}
		Since $\Gamma_{1}\square \Gamma_{2}$ is $(R_1+R_2)-$regular and from the definition of normalized connection Laplacian, we have
		\begin{equation}\label{eqL1}
			\begin{aligned}
				&\mathcal{L}^{\hat{\sigma}} \delta_{i_1}^{j_1}\otimes \chi_{i_2}^{j_2}(x,y)\\
				&=\delta_{i_1}^{j_1}\otimes \chi_{i_2}^{j_2}(x,y)-\dfrac{1}{R_1+R_2}\left[\sum_{x^{'}\sim x}\left(\sigma^{(1)}_{xx^{'}}\otimes I_{d_2}\right)\left(\delta_{i_1}^{j_1}(x^{'})\otimes \chi_{i_2}^{j_2}(y)\right)\right.\\	
				&+\left.\sum_{y^{'}\sim y}\left(I_{d_1}\otimes \sigma^{(2)}_{yy^{'}}\right)\left(\delta_{i_1}^{j_1}(x)\otimes \chi_{i_2}^{j_2}(y^{'})\right)\right]			
			\end{aligned}
		\end{equation}
		According to the mixed-product property and bilinearity of "$\otimes$", we have 
		\begin{equation}\label{eqL2}
			\begin{aligned}
				\sum_{x^{'}\sim x}\left(\sigma^{(1)}_{xx^{'}}\otimes I_{d_2}\right)\left(\delta_{i_1}^{j_1}(x^{'})\otimes \chi_{i_2}^{j_2}(y)\right)&=\left(\sum_{x^{'}\sim x}\sigma^{(1)}_{xx^{'}}\delta_{i_1}^{j_1}(x^{'})\right)\otimes \chi_{i_2}^{j_2}(y)\\	
				\sum_{y^{'}\sim y}\left(I_{d_1}\otimes \sigma^{(2)}_{yy^{'}}\right)\left(\delta_{i_1}^{j_1}(x)\otimes \chi_{i_2}^{j_2}(y^{'})\right)&=\delta_{i_1}^{j_1}(x)\otimes \left(\sum_{y^{'}\sim y}\sigma^{(2)}_{yy^{'}}\chi_{i_2}^{j_2}(y^{'})\right)
			\end{aligned}
		\end{equation}
		Dividing $\delta_{i_1}^{j_1}\otimes \chi_{i_2}^{j_2}$ into $\frac{R_1}{R_1+R_2}\delta_{i_1}^{j_1}\otimes \chi_{i_2}^{j_2}$ and $\frac{R_2}{R_1+R_2}\delta_{i_1}^{j_1}\otimes \chi_{i_2}^{j_2}$ as well as putting equation (\ref{eqL2}) into equation (\ref{eqL1}), we get
		\begin{align*}
			\mathcal{L}^{\hat{\sigma}} \delta_{i_1}^{j_1}\otimes \chi_{i_2}^{j_2}(x,y)
			&=\frac{R_1}{R_1+R_2}\left[\left(\delta_{i_1}^{j_1}(x)-\frac{1}{R_1}\sum_{x^{'}\sim x}\sigma^{(1)}_{xx^{'}}\delta_{i_1}^{j_1}(x^{'})\right)\otimes \chi_{i_2}^{j_2}(y) \right]\\
			&+\frac{R_2}{R_1+R_2}\left[\delta_{i_1}^{j_1}(x)\otimes \left(\chi_{i_2}^{j_2}(y)-\frac{1}{R_2}\sum_{y^{'}\sim y}\sigma^{(2)}_{yy^{'}}\chi_{i_2}^{j_2}(y^{'})\right)\right]\\
			&=\frac{R_1}{R_1+R_2}\mathcal{L}^{\sigma^{(1)}}\delta_{i_1}^{j_1}(x)\otimes \chi_{i_2}^{j_2}(y)+
			\frac{R_2}{R_1+R_2}\delta_{i_1}^{j_1}(x)\otimes \mathcal{L}^{\sigma^{(2)}}\chi_{i_2}^{j_2}(y)
		\end{align*}
	\end{proof}
	
	\begin{theorem}\label{thm1}
		\[ \mathcal{L}^{\hat{\sigma}}=\left(\frac{R_1}{R_1+R_2}\mathcal{L}^{\sigma^{(1)}}\right)\oplus \left(\frac{R_2}{R_1+R_2}\mathcal{L}^{\sigma^{(2)}}\right) \]
	\end{theorem}
	
	\begin{proof}
		For any function $F\in C((\Gamma_{1}\square \Gamma_{2},\hat{\sigma}),\mathbb{R}^{d_1d_2}) $ $$F=\sum_{i_1,i_2,j_1,j_2}c_{i_1,i_2}^{j_1,j_2}\delta_{i_1}^{j_1}\otimes \chi_{i_2}^{j_2}$$ As a result of Lemma \ref{lemma1} and the linearity of $\mathcal{L}^{\hat{\sigma}}$ , we have \[ \mathcal{L}^{\hat{\sigma}}F= \left(\frac{R_1}{R_1+R_2}\mathcal{L}^{\sigma^{(1)}}\otimes I_{n_2d_2}+\frac{R_2}{R_1+R_2}I_{n_1d_1}\otimes \mathcal{L}^{\sigma^{(2)}}\right)F\]
		Therefore, \[ \begin{aligned}
			\mathcal{L}^{\hat{\sigma}}&= \frac{R_1}{R_1+R_2}\mathcal{L}^{\sigma^{(1)}}\otimes I_{n_2d_2}+\frac{R_2}{R_1+R_2}I_{n_1d_1}\otimes \mathcal{L}^{\sigma^{(2)}}\\
			&=\left(\frac{R_1}{R_1+R_2}\mathcal{L}^{\sigma^{(1)}}\right)\oplus \left(\frac{R_2}{R_1+R_2}\mathcal{L}^{\sigma^{(2)}}\right) 
		\end{aligned}
		\]
	\end{proof}
	
	In the aforementioned discussion, we define product connection only on the Cartesian product of two connection graphs. In fact, $\forall m\in \mathbb{Z}^{+}$, product connection can be defined on the Cartesian product of $m$ connection graphs.
	\begin{definition}
		Suppose $\{(\Gamma_{i},\sigma^{(i)})\}_{i=1}^{m}$ are connection graphs with $\sigma^{(i)}:E(\Gamma_{i})\rightarrow O(d_i)$. Let $\Gamma_{1} \square \Gamma_{2}\square \cdots\square \Gamma_{m} $ be the Cartesian product graph of $\{\Gamma_{i}\}_{i=1}^{m}$ . Define a connection $\hat{\sigma}:E(\Gamma_{1} \square \Gamma_{2}\square\cdots\square \Gamma_{m})\rightarrow O(d_1d_2\cdots d_m)$ as follows:
		\begin{equation*}
			\begin{aligned}
				&\hat{\sigma}_{(x_1,x_2,\cdots,x_i,\cdots,x_m)(x_1,x_2,\cdots,y_i,\cdots,x_m)}	\\
				&=I_{d_1\times d_1}\otimes  \cdots\otimes I_{d_{i-1}\times d_{i-1}}\otimes \sigma^{(i)}_{x_iy_i}\otimes I_{d_{i+1}\times d_{i+1}}\otimes \cdots \otimes I_{d_m\times d_m}\\
			\end{aligned}
		\end{equation*} 
		We call $\hat{\sigma}$ product connection on $\Gamma_{1} \square \Gamma_{2}\square \cdots\square \Gamma_{m} $, which is denoted by $\sigma^{(1)}\otimes\cdots\otimes \sigma^{(m)}$	.
	\end{definition}
	
	Let $ \mathcal{L}^{\hat{\sigma}}$ be the normalized connection Laplacian of the Cartesian product $\Gamma_{1} \square \cdots\square \Gamma_{m}$ with product connection $\hat{\sigma}$. Assume every $ \Gamma_i$ is $R_i-$regular and has simple weight. The derivation of the subsequent theorem is very similar to the proof of Theorem \ref{thm1} so its proof is omitted. 
	
	\begin{theorem}\label{n-composition}
		\[ 
		\mathcal{L}^{\hat{\sigma}}=\left(\frac{R_1}{R_1+\cdots+R_m}\mathcal{L}^{\sigma^{(1)}}\right)\oplus \cdots \oplus \left(\frac{R_m}{R_1+\cdots+R_m}\mathcal{L}^{\sigma^{(n)}}\right)
		\]
	\end{theorem}
	
	Afterwards, we derive the expression of connection heat kernel on $\mathbb{Z}^n$ with product connection $\hat{\sigma}$. Suppose $\left\{ (\mathbb{Z},\sigma^{(i)})\right\}_{i=1}^n$ are connection graphs and $\hat{\sigma}=\sigma^{(1)}\otimes \sigma^{(2)}\otimes\cdots\otimes\sigma^{(n)}$. Let $\mathcal{L}^{\mathbb{Z}^n,\hat{\sigma}} $ be the normalized connection Laplacian of $(\mathbb{Z}^n,\hat{\sigma})$ and $H_t^{\mathbb{Z}^n,\hat{\sigma}} $ be the connection heat kernel on $(\mathbb{Z}^n,\hat{\sigma})$. 
	
	To enhance the elegance of the expression of $H_t^{\mathbb{Z}^n,\hat{\sigma}} $, we introduce one concept called signature in connection graph.
	\begin{definition}
		Let $(\Gamma,\sigma)$ be a connection graph. For any path $P: x_0\sim x_1\sim \cdots \sim x_n$ , the signature of $P$ is defined as follows:\[ \sigma_{P}=\sigma_{x_0x_1}\sigma_{x_1x_2}\cdots\sigma_{x_{n-1}x_n} \] 	
	\end{definition}
	
	Assume $x$ and $y$ are two vertices in $(\mathbb{Z},\sigma)$, there exists only one path $P$ from $x$ to $y$. We denote the signature of the unique path by $\sigma_{P_{x\rightarrow y}}$ and say $\sigma_{P_{x\rightarrow y}}$ is the signature from $x$ to $y$ in $(\mathbb{Z},\sigma)$.
	
	\begin{theorem}
		\begin{equation*}\label{equZ}
			\begin{aligned}
				&H_t^{\mathbb{Z}^n,\hat{\sigma}}\left(\left(x_1,x_2,\cdots,x_n\right),\left(x_1+a_1,x_2+a_2,\cdots,x_n+a_n\right)\right)\\
				&=\prod_{i=1}^{n}\left((-1)^{|a_i|} \sum_{k\ge 0}\frac{C_{2k}^{k+|a_i|}}{k!} (-\frac{t}{2n})^k\right) \sigma^{(1)}_{P_{x_1\rightarrow x_1+a_1}}\otimes \cdots\otimes \sigma^{(n)}_{P_{x_n\rightarrow x_n+a_n}}
			\end{aligned}
		\end{equation*}
		where $ (x_1,x_2,\cdots,x_n),(a_1,\cdots,a_n)\in \mathbb{Z}^n $ and $ \sigma^{(i)}_{P_{x_i\rightarrow x_i+a_i}}$ is the signature from $x_i$ to $x_i+a_i$ in $(\mathbb{Z},\sigma^{(i)})$.
	\end{theorem}
	
	\begin{proof}
		Due to Theorem \ref{n-composition}, we have $$\mathcal{L}^{\mathbb{Z}^n,\hat{\sigma}}=\left(\frac{1}{n}\mathcal{L}^{\sigma^{(1)}}\right)\oplus \left(\frac{1}{n}\mathcal{L}^{\sigma^{(2)}}\right)\oplus \cdots \oplus \left(\frac{1}{n}\mathcal{L}^{\sigma^{(n)}}\right)$$
		Therefore,\begin{equation*}
			\begin{aligned}
				H_t^{\mathbb{Z}^n,\hat{\sigma}}&=exp(-t\mathcal{L}^{\mathbb{Z}^n,\hat{\sigma}})\\
				&=exp\left[\left(-\frac{t}{n}\mathcal{L}^{\sigma^{(1)}}\right)\oplus \left(-\frac{t}{n}\mathcal{L}^{\sigma^{(2)}}\right)\oplus \cdots \oplus \left(-\frac{t}{n}\mathcal{L}^{\sigma^{(n)}}\right)\right]\\
				&=exp(-\frac{t}{n}\mathcal{L}^{\sigma^{(1)}})\otimes exp(-\frac{t}{n}\mathcal{L}^{\sigma^{(2)}})\otimes\cdots \otimes exp(-\frac{t}{n}\mathcal{L}^{\sigma^{(n)}})\\
				&=H_{\frac{t}{n}}^{\sigma^{(1)}}\otimes H_{\frac{t}{n}}^{\sigma^{(2)}}\otimes \cdots\otimes H_{\frac{t}{n}}^{\sigma^{(n)}}
			\end{aligned}
		\end{equation*} 
		and \begin{equation*}
			\begin{aligned}
				&H_t^{\mathbb{Z}^n,\hat{\sigma}}\left(\left(x_1,x_2,\cdots,x_n\right),\left(x_1+a_1,x_2+a_2,\cdots,x_n+a_n\right)\right)\\
				&=H_{\frac{t}{n}}^{\sigma^{(1)}}(x_1,x_1+a_1)\otimes H_{\frac{t}{n}}^{\sigma^{(2)}}(x_2,x_2+a_2)\otimes \cdots\otimes H_{\frac{t}{n}}^{\sigma^{(n)}}(x_n,x_n+a_n)\\
			\end{aligned}
		\end{equation*}
		Putting the results of Theorem \ref{Zexpre} into the above expression, we hereby finish the proof.
	\end{proof}
	
	\section{The Connection Heat Kernel On Consistent Graph}
	
	\begin{definition}
		If the signature of every cycle in $\Gamma$ is equal to the identity matrix,  we call $(\Gamma,\sigma)$ a consistent graph and we say $\sigma$ is a balanced connection.
	\end{definition}
	
	\begin{remark}
		Assume $(\Gamma,\sigma)$ is consistent graph, if $P_1$ and $P_2$ are two paths which have same starting point and ending point, then the signature of $P_1$ is the same as the signature of $P_2$.
	\end{remark}
	
	More information about the properties and equivalent definition of consistent graph can be found in \cite{chung2013local,chung2014ranking}. 
	
	Combining the findings in the previous section with the expression of heat kernel on lattices without connection in \cite{chung1997combinatorial}, we observe that connection heat kernel on $(\mathbb{Z}^n,\hat{\sigma})$ is equal to certain signature multiplied by the heat kernel of $\mathbb{Z}^n$ .  In this section, the following theorem elucidates that the connection heat kernel $H_t^{\sigma}(x,y)$ on any consistent graph $(\Gamma,\sigma) $ equals the signature of any path from $x$ to $y$ multiplied by the heat kernel $H_t(x,y)$ of $\Gamma$, which is verified in different ways depending on whether $\Gamma$ is finite or infinite. 
	
	\begin{theorem}\label{consisthm}
		For any consistent graph $(\Gamma,\sigma)$:
		\[ H_t^{\sigma}(x,y)=H_t(x,y)\sigma_{P_{x\rightarrow y}} \] 
		where $H_t^{\sigma} $ is connection heat kernel on $(\Gamma,\sigma)$, $H_t(x,y)$ is heat kernel on underlying graph $\Gamma$ and $\sigma_{P_{x\rightarrow y}}$ is the signature from $x$ to $y$ in $(\Gamma,\sigma)$.
	\end{theorem}
	
	\textbf{One case: $\Gamma$ is finite}
	
	Here we assume $\Gamma$ is a finite graph with $n$ vertices and $\sigma: E(\Gamma)\rightarrow O(d)$ is balanced. Without loss of generality, suppose $\Gamma$ is connected. Otherwise, replace $\Gamma$ by the connected component of $x$ and $y$ in the proof. Since $(\Gamma,\sigma) $ is consistent, we can construct a special eigensystem of $\mathcal{L}^{\sigma}$. 
	
	Let $0=\lambda_1<\lambda_2\le \cdots\le \lambda_n $
	be the eigenvalues of $\mathcal{L}$. Let $\mu_1\le \mu_2 \cdots \le \mu_{nd}$ be the eigenvalues of $\mathcal{L}^{\sigma}$. According to the spectrum of consistent graph in \cite[Theorem 1]{chung2014ranking}, we have $ \lambda_1=\mu_1=\mu_2=\cdots=\mu_d;\lambda_2=\mu_{d+1}=\cdots=\mu_{2d};\cdots; \lambda_n=\mu_{(n-1)d+1}=\cdots=\mu_{nd}$. Choose $\{f_i:V(\Gamma)\rightarrow \mathbb{R}\}_{i=1,\cdots,n}$ as orthonormal eigenfunctions of normalized standard Laplacian $\mathcal{L}$ with respect to $\{\lambda_i\}_{i=1,\cdots,n}$.  
	
	Choosing a fixed vertex $x_1$, define $\{g_j:V(\Gamma)\rightarrow \mathbb{R}^d\}_{j=1,\cdots,d}$ as follows:
	\begin{equation*}
		\begin{aligned}
			g_j(x_1)&=e_j\\
			g_j(y)&=\sigma_{P_{y\rightarrow x_1}}g_j(x_1)
		\end{aligned}
	\end{equation*}
	where $e_j=[0,\cdots,0,\underbrace{1}_{j-th},0,\cdots,0]^T$ is a unit vector in $\mathbb{R}^d$ , $P_{y\rightarrow x_1} $ is one path from $y$ to $x_1$ and $ \sigma_{P_{y\rightarrow x_1}}$ is the signature of $P_{y\rightarrow x_1} $. 
	
	By $\mathcal{L}^{\sigma}g_j(x)=\frac{1}{deg(x)}\sum_{y\sim x} w_{xy}(g_j(x)-\sigma_{xy}g_j(y))=0$ and $g_i\perp g_j, \forall i\neq j$ , we know $\{g_j\}_{j=1,\cdots,d}$ are orthogonal eigenfunctions of $\mathcal{L}^{\sigma}$ with respect to $\{\mu_i\}_{i=1,\cdots,d}$. 
	
	For $\forall i=1,\cdots,n;j=1,\cdots,d $, define $\Phi_{i,j}$ as follows:
	\[ \Phi_{i,j}:V(\Gamma)\rightarrow \mathbb{R}^d \]
	\[ \Phi_{i,j}(x)=f_i(x)g_j(x) \]
	
	Regard $\Phi_{i,j}$ as a vector in $\mathbb{R}^{nd}$ and $||\Phi_{i,j}||_2=1$. Due to $f_{i_1}\perp f_{i_2}$ if $i_1\neq i_2$ and $g_{j_1}\perp g_{j_2}$ if $j_1\neq j_2$, $\Phi_{i_1,j_1}$ is orthogonal to $\Phi_{i_2,j_2}$ if $(i_1,j_1)\neq (i_2,j_2)$.  
	We claim $\Phi_{i,j}$ is the eigenfunction of $\mathcal{L}^{\sigma}$ with respect to $\mu_{(i-1)d+j}$. In fact, 
	\begin{equation*}
		\begin{aligned}
			\mathcal{L}^{\sigma}\Phi_{i,j}(x)&=\frac{1}{deg(x)}\sum_{y\sim x} w_{xy}(\Phi_{i,j}(x)-\sigma_{xy}\Phi_{i,j}(y))\\
			&=\frac{1}{deg(x)}\sum_{y\sim x} w_{xy}(f_i(x)g_j(x)-\sigma_{xy}f_i(y)g_j(y))\\
			&=\frac{1}{deg(x)}\sum_{y\sim x} w_{xy}(f_i(x)g_j(x)-f_i(y)g_j(x))\\
			&=\frac{1}{deg(x)}\sum_{y\sim x} w_{xy}(f_i(x)-f_i(y)) g_j(x)\\
			&=\mathcal{L}f_i(x)g_j(x)\\
			&=\mu_{(i-1)d+j}f_i(x)g_j(x)\\
			&=\mu_{(i-1)d+j} \Phi_{i,j}(x)
		\end{aligned}
	\end{equation*}
	Therefore, $\{(\Phi_{i,j},\mu_{(i-1)d+j})\}_{i=1,\cdots,n;j=1,\cdots,d}$ is the orthonormal system of $\mathcal{L}^{\sigma}$.

	\begin{proof}[The Proof of Theorem \ref{consisthm}]
		\begin{equation}\label{eq1}
			\begin{aligned}
				H_t^{\sigma}(x,y)&=\sum_{i=1}^{n}\sum_{j=1}^{d}exp(-t\mu_{(i-1)d+j})\Phi_{i,j}(x)\overline{\Phi_{i,j}(y)}^T\\
				&=\left[\sum_{i=1}^{n}exp(-t\lambda_i) f_i(x) f_i(y)\right] \left[\sum_{j=1}^{d}	g_j(x)\overline{g_j(y)}^T\right] \\
				&=H_t(x,y)\left[\sum_{j=1}^{d}	g_j(x)\overline{g_j(y)}^T\right] 
			\end{aligned}
		\end{equation}
		Each $g_i:V(\Gamma)\rightarrow \mathbb{R}^d$ can be thought of as a vector in $\mathbb{R}^{nd} $ and let $S:=\sum_{j=1}^{d}	g_j\overline{g_j}^T$. Then $S$ is a $nd\times nd$ matrix and each block $S(x,y)= \sum_{j=1}^{d}	g_j(x)\overline{g_j(y)}^T$ is a $d\times d$ matrix. Let $P_{x\rightarrow x_1}$ be a path from $x $ to $x_1$.  Let $P_{y\rightarrow x_1}$ be a path from $y $ to $x_1$.
		\begin{equation}\label{eq2}
			\begin{aligned}
				S(x,y)&=\sum_{j=1}^{d}\sigma_{P_{x\rightarrow x_1}}g_j(x_1)\left(\overline{\sigma_{P_{y\rightarrow x_1}}g_j(x_1)}\right)^T\\
				&=\sum_{j=1}^{d}\sigma_{P_{x\rightarrow x_1}}e_je_j^T\sigma_{P_{x_1\rightarrow y}}\\
				&=\sigma_{P_{x\rightarrow x_1}}(\sum_{j=1}^{d}e_je_j^T)\sigma_{P_{x_1\rightarrow y}}\\
				&=\sigma_{P_{x\rightarrow x_1}}\sigma_{P_{x_1\rightarrow y}}\\
				&=\sigma_{P_{x\rightarrow y}}
			\end{aligned}
		\end{equation}
		Put the expression (\ref{eq2}) into (\ref{eq1}),we get 
		\[ H_t^{\sigma}(x,y)=H_t(x,y)\sigma_{P_{x\rightarrow y}} \]
	\end{proof}
	
	\textbf{The other case: $\Gamma$ may be infinite}
	
	Here we assume $\Gamma$ is a locally finite graph and $\sigma: E(\Gamma)\rightarrow O(d)$ is balanced.
	It's obvious that $\mathcal{L}^k(x,y)$ is the sum of weights of all walks of length $k$ from $x$ to $y$ for any $k\in \mathbb{N}$ and vertices $ x,y$. Since $\sigma$ is balanced, all walk from $x$ to $y$ have same signature $\sigma_{P_{x\rightarrow y}}$, implying $(\mathcal{L}^{\sigma})^k(x,y)=\mathcal{L}^k(x,y)\sigma_{P_{x\rightarrow y}}$.
	
	\begin{proof}[The Proof of Theorem \ref{consisthm}]
		First, we do the Taylor expansion on $H_t^{\sigma}$	.
		\begin{equation*}
			\begin{aligned}
				H_t^{\sigma}(x,y)&=e^{-t\mathcal{L}^{\sigma}}(x,y)\\
				&=\left[I-t\mathcal{L}^{\sigma}+\frac{1}{2!}t^2(\mathcal{L}^{\sigma})^2+\cdots+\frac{1}{k!}(-t)^k(\mathcal{L}^{\sigma})^k+\cdots\right](x,y)
			\end{aligned}
		\end{equation*}
		Due to $(\mathcal{L}^{\sigma})^k(x,y)=\mathcal{L}^k(x,y)\sigma_{P_{x\rightarrow y}}$, we have
		\begin{equation*}
			\begin{aligned}
				H_t^{\sigma}(x,y)&=	\left[I-t\mathcal{L}+\frac{1}{2!}t^2(\mathcal{L})^2+\cdots+\frac{1}{k!}(-t)^k(\mathcal{L})^k+\cdots\right](x,y)\sigma_{P_{x\rightarrow y}}\\
				&=e^{-t\mathcal{L}}(x,y)\sigma_{P_{x\rightarrow y}}\\
				&=H_t(x,y)\sigma_{P_{x\rightarrow y}}
			\end{aligned}
		\end{equation*}
	\end{proof}
	
	\section{The Connection Heat Kernel On Connection Discrete Torus}
	
	Let $M$ be an integer $n\times n$ matrix with $det M>1$. We consider $M\mathbb{Z}^n$ as an additive group acting on $\mathbb{Z}^n$, then $\mathbb{Z}^n/M\mathbb{Z}^n $ is a quotient group of $\mathbb{Z}^n$. Now we regard $\mathbb{Z}^n/M\mathbb{Z}^n $ as the quotient graph of $\mathbb{Z}^n$, called discrete torus. It is a finite graph with $|V(\mathbb{Z}^n/M\mathbb{Z}^n)|=det M$. 
	
	Given a connection graph and a proper group action on it, we introduce the concept of quotient connection graph in this section. Then we define connection discrete torus as the quotient of connection lattice $(\mathbb{Z}^n,\sigma)$ and additive group action $M\mathbb{Z}^n$. Furthermore, we investigate the relation between the connection heat kernel on any connection graph and connection heat kernel on its quotient connection graph, from which we acquire the expression of connection heat kernel on connection discrete torus.
	
	\begin{definition}
		$g$ is called an automorphism on connection graph $(\Gamma,\sigma)$, if the following are satisfied:
		\begin{itemize}
			\item $gx\sim gy $ iff $x\sim y ,\forall x,y\in V(\Gamma)$.
			\item $w_{gx,gy}=w_{x,y}, \forall x,y\in V(\Gamma)$ .
			\item $\sigma_{gx,gy}=\sigma_{xy}, \forall x,y\in V(\Gamma)$ .
		\end{itemize}
		The set of all the automorphisms on $(\Gamma,\sigma)$ is denoted by $Aut(\Gamma)$.
	\end{definition}
	
	\begin{remark}
		The degree of $x$ is invariant under automorphism because $d(x)=\sum_{y\sim x} w_{xy}=\sum_{y\sim x} w_{gx,gy}=\sum_{z\sim gx} w_{gx,z}=d(gx)$.
	\end{remark}
	
	\begin{definition}
		Given a group $G$ acting on the connection graph $(\Gamma,\sigma)$, we say the connection $\sigma$ is $G-$proper if the following is satisfied:
		\[ \forall [x]\neq [y] , \forall v\in [x],\forall w\in [y], \sigma_{vw}=\sigma_{xy}\]
		where $[x]=\{gx|g\in G\},[y]=\{gy|g\in G\}$ are the equivalent classes of vertices under the action of $G$.
	\end{definition}
	
	\begin{definition}
		Suppose $(\Gamma,\sigma) $ is a connection graph and $G$ is a subgroup of $Aut(\Gamma)$. If $\sigma$ is $G-$proper, then we can define a connection quotient graph $(\Gamma /G,\tilde{w},\sigma^{Q_G})$ as follows:
		\begin{itemize}
			\item The vertices in $\Gamma /G $ are the equivalent classes $[x]$ under $G$, where $[x]:=\{y|\exists g\in G,y=gx\}$.
			\item $[x]\sim [y]$ if and only if $\exists v\in [x], \exists w\in [y], v\sim w$ and $[x]\cap [y]=\emptyset$.
			\item  $\tilde{w}_{[x],[y]}:=\sum_{g\in G} w_{x,gy},\forall [x]\sim [y]$.
			\item $\sigma^{Q_G}_{[x],[y]}:=\sigma_{xy}, \forall [x]\sim [y]$.
		\end{itemize}
		We call $\sigma^{Q_G}$ as the quotient connection with respect to the group $G$.
	\end{definition}
	
	In our paper, we say $(\mathbb{Z}^n/M\mathbb{Z}^n,\beta)$ is a connection discrete torus if there exists a $M\mathbb{Z}^n-$proper connection $\sigma$ on $\mathbb{Z}^n$ such that $\beta=\sigma^{Q_{M\mathbb{Z}^n}}$. In other words, we refer to connection discrete torus as a quotient of a connection lattice $(\mathbb{Z}^n,\sigma)$.
	
	Before acquiring the expression of connection heat kernel on $(\mathbb{Z}^n/M\mathbb{Z}^n,\sigma^{Q_{M\mathbb{Z}^n}})$, we study how a group action on connection graph affects its connection heat kernel firstly, which is shown in next lemma.
	
	\begin{lemma}
		The connection heat kernel $H_t^{\sigma}$ on $(\Gamma,\sigma :E(\Gamma)\rightarrow O(d)) $ is invariant under $Aut(\Gamma)$:
		\[ H_t^{\sigma}(x,y)=H_t^{\sigma}(gx,gy), \forall x,y\in V(\Gamma), \forall g \in Aut(\Gamma) \]
	\end{lemma}
	
	\begin{proof}
		First we claim that:
		
		For $\forall g\in Aut(\Gamma), \forall f\in C((\Gamma,\sigma),\mathbb{R}^d)$,
		$ \mathcal{L}^{\sigma}(f\circ g)|_{x}=\mathcal{L}^{\sigma} f|_{gx} $.
		
		In fact:
		\begin{equation*}
			\begin{aligned}
				\mathcal{L}^{\sigma}(f\circ g)|_{x}&=f(gx)-\frac{1}{d(x)}\sum_{y\sim z}w_{xy}\sigma_{xy}f(gy)\\
				&=f(gx)-\frac{1}{d(gx)}\sum_{y\sim x}w_{gx,gy}\sigma_{gx,gy}f(gy)\\
				&=f(gx)-\frac{1}{d(gx)}\sum_{z\sim gx}w_{gx,z}\sigma_{gx,z}f(z)\\
				&=\mathcal{L}^{\sigma} f|_{gx}
			\end{aligned}
		\end{equation*}
		
		For $1\le i\le d $, define $\delta_y^{i}: V(\Gamma)\rightarrow \mathbb{R}^d$ as follows:
		\begin{equation*}
			\delta_y^{i}(x)=\left\{	\begin{aligned}
				e_{i}& \quad  x=y\\
				\vec{0}&\quad else
			\end{aligned}\right.
		\end{equation*}
		where $e_{i} $ is a unit vector in $\mathbb{R}^d$ and its $i$th entry equals $1$.
		
		Using the above claim,
		\begin{equation*}
			\begin{aligned}
				\mathcal{L}^{\sigma}\delta_{gy}^{i}|_{gx}&=\mathcal{L}^{\sigma}(\delta_{gy}^{i}\circ g)|_{x}\\
				&=\mathcal{L}^{\sigma}\delta_{y}^{i}|_{x}
			\end{aligned}
		\end{equation*}
		Since $ \mathcal{L}^{\sigma}\delta_{gy}^{i}|_{gx}$ is the $i$th column of $H_t^{\sigma}(gx,gy) $, $\mathcal{L}^{\sigma}\delta_{y}^{i}|_{x} $ is the $i$th column of $H_t^{\sigma}(x,y) $ and $i$ takes the integers from $1$ to $d$, we know $H_t^{\sigma}(gx,gy)=H_t^{\sigma}(x,y) $ .
	\end{proof}
	
	We say a function $f\in C((\Gamma,\sigma),\mathbb{R}^d)$ is $G-$periodic if $ f(gx)=f(x),\forall x\in V(\Gamma),\forall g\in G$. A $G-$periodic function on $(\Gamma,\sigma)$ can be regarded as a function on $(\Gamma/G,\sigma^{Q_G})$ while a function on $(\Gamma/G,\sigma^{Q_G})$ can be extended to a $G-$periodic function on $(\Gamma,\sigma)$.
	
	\begin{lemma}\label{solutionlemma}
		If $f$ is $G-$periodic on $(\Gamma,\sigma)$, then 
		\[ \mathcal{L}^{\sigma}f(x)=\mathcal{L}^{\sigma^{Q_G}}f([x]) \]
	\end{lemma}
	
	\begin{proof}
		\begin{equation*}
			\begin{aligned}
				\mathcal{L}^{\sigma^{Q_G}}f([x])&=f([x])-\frac{1}{d([x])}\sum_{[y]\in \Gamma/G}w_{[x],[y]}\sigma_{[x],[y]}^{G}f([y])\\
				&=f(x)-\frac{1}{d(x)}\sum_{[y]\in \Gamma/G}\left(\sum_{g\in G}w_{x,gy}\right)\sigma_{xy}f(y)\\
				&=f(x)-\frac{1}{d(x)}\sum_{[y]\in \Gamma/G}\left(\sum_{g\in G}w_{x,gy}\right)\sigma_{x,gy}f(gy)\\
				&=f(x)-\frac{1}{d(x)}\sum_{z\in \Gamma}w_{xz}\sigma_{xz}f(z)\\
				&=\mathcal{L}^{\sigma}f(x)
			\end{aligned}
		\end{equation*}
	\end{proof}
	
	We derive the expression of connection heat kernel on a quotient connection graph in the following theorem.
	
	\begin{theorem}\label{sumheat}
		If $\sigma$ is $G-$proper, then the connection heat kernel on  $(\Gamma,\sigma) $ and the connection heat kernel on $(\Gamma/G,\sigma^{Q_G})$ have the following relation:
		\[ H_t^{\sigma^{Q_G}}([x],[y])=\sum_{g\in G} H_t^{\sigma}(x,gy) \]
	\end{theorem}
	
	\begin{proof}
		From the above lemma \ref{solutionlemma}, we know that a $G-$periodic solution of the heat equation with the $G-$periodic initial condition on $(\Gamma,\sigma)$ is the unique solution of heat equation on $(\Gamma/G,\sigma^{Q_G})$. 
		
		Suppose $f$ is the initial condition on $(\Gamma/G,\sigma^{Q_G})$, which can be regarded as a $G-$periodic initial condition on $(\Gamma,\sigma)$, then the solution of the heat equation on $(\Gamma,\sigma)$ is
		\begin{equation*}
			\begin{aligned}
				u(x,t)&=\sum_{z\in V(\Gamma)}H_t^{\sigma}(x,z)f(z)\\
				&=\sum_{[y]\in V(\Gamma/G)}\sum_{g\in G}H_t^{\sigma}(x,gy)f(gy)\\
				&=\sum_{[y]\in V(\Gamma/G)}\left(\sum_{g\in G}H_t^{\sigma}(x,gy)\right)f([y])
			\end{aligned}
		\end{equation*}
		Then $H_t^{\sigma^{Q_G}}([x],[y])=\sum_{g\in G} H_t^{\sigma}(x,gy) $.
	\end{proof}
	
	Just from theorem \ref{equZ} and theorem \ref{sumheat}, we can acquire the follwing theorem describing expression of connection heat kernel on connection discrete torus.
	\begin{theorem}\label{torus 1}
		Suppose $(\mathbb{Z}^n,\hat{\sigma})$ is product connection graph where $\hat{\sigma}$ is $M\mathbb{Z}^n-$proper and $\hat{\sigma}=\sigma^{(1)}\otimes \cdots\otimes\sigma^{(n)}$. Then  connection heat kernel of $(\mathbb{Z}^n/M\mathbb{Z}^n,\hat{\sigma}^{Q_{M\mathbb{Z}^n}})$ is:
		\[ 
		\begin{aligned}
			&H_t^{\hat{\sigma}^{Q_{M\mathbb{Z}^n}}}([x],[y])\\
			&=\sum_{a\in M\mathbb{Z}^n}\prod_{i=1}^{n}\left((-1)^{|y_i+a_i-x_i|} \sum_{k\ge 0}\frac{C_{2k}^{k+|y_i+a_i-x_i|}}{k!} (-\frac{t}{2n})^k\right)\\ &\sigma^{(1)}_{P_{x_1\rightarrow y_1+a_1}}\otimes \sigma^{(2)}_{P_{x_2\rightarrow y_2+a_2}}\otimes\cdots\otimes \sigma^{(n)}_{P_{x_n\rightarrow y_n+a_n}}
		\end{aligned}
		\]
	\end{theorem}
	
	\section{An Equation on Connection Heat Kernel of Discrete Torus}
	In the preceding section, we have derived expressions for the connection heat kernel on connection discrete torus. However, we may directly compute the block entry of the connection heat kernel from its characteristic representation if it is easy to ascertain the spectrum of the normalized connection Laplacian of connection discrete torus.
	From lemma \ref{solutionlemma}, a $M\mathbb{Z}^n-$periodic eigenfunction of normalized connection Laplacian $\mathcal{L}^{\sigma}$ on $(\mathbb{Z}^n,\hat{\sigma})$ is an eigenfuntion of normalized connection Laplacian $\mathcal{L}^{\hat{\sigma}^{Q_{M\mathbb{Z}^n}}}$ on $(\mathbb{Z}^n/M\mathbb{Z}^n,\hat{\sigma}^{Q_{M\mathbb{Z}^n}})$. Therefore, we seek  $M\mathbb{Z}^n-$periodic eigenfunctions of $\mathcal{L}^{\sigma}$ on $(\mathbb{Z}^n,\hat{\sigma})$ firstly.
	
	We say a connection $\sigma:E(\Gamma)\rightarrow O(d)$ is a constant connection if  $\sigma(E)=\{\sigma_1,\sigma_1^{-1}\}$ where $\sigma_1$ is a matrix in $O(d)$. We say the product connection $\hat{\sigma}:=\sigma^{(1)}\otimes \cdots\otimes\sigma^{(n)}$ is a constant product connection if each component connection $\sigma^{(i)}$ is a constant connection. In fact, when the connection on $\mathbb{Z}^n$ is a constant connection, it's not difficult to obtain the spectrum of normalized connection Laplacian.   
\begin{lemma}
	Suppose $(\mathbb{Z},\sigma_1:E(\mathbb{Z})\rightarrow O(d))$ is connection graph and $\sigma_1$ is a constant connection. Let $\{(\lambda_k,v_k)\}_{k=1}^{d} $ be the orthonormal eigensystem of $\sigma_1$. $\forall w \in \mathbb{R}, 1\le k\le d$, define $f_{w,k}:V(\mathbb{Z})\rightarrow \mathbb{R}^d$ as $f_{w,k}(x)=e^{2\pi iwx }v_k$. Then $f_{w,k} $ is an eigenfunction of $\mathcal{L}^{\sigma_1}$ with respect to $1-\frac{1}{2}\lambda_k^{-1} e^{-2\pi iw }-\frac{1}{2}\lambda_k e^{2\pi iw } $.
\end{lemma}

\begin{proof}
	Since $\sigma_1 v_k=\lambda_k v_k$ and $(\sigma_1)^{-1} v_k=(\lambda_k)^{-1} v_k$ , we have
	\begin{equation*}
		\begin{aligned}
			&\mathcal{L}^{\sigma_1}f_{w,k}(x)
			=e^{2\pi iwx }v_k-\frac{1}{2}\left[\sigma_1^{-1} e^{2\pi iw(x-1)}v_k
			+ \sigma_1 e^{2\pi iw(x+1)} v_k\right]\\
			&=\left[e^{2\pi iwx }-\frac{1}{2}\left(\lambda_k^{-1} e^{2\pi iw(x-1)}+ \lambda_k e^{2\pi iw(x+1)}\right)\right]v_k\\
			&=\left(1-\frac{1}{2}\lambda_k^{-1} e^{-2\pi iw }-\frac{1}{2}\lambda_k e^{2\pi iw }\right)e^{2\pi iwx }v_k\\
			&=\left(1-\frac{1}{2}\lambda_k^{-1} e^{-2\pi iw }-\frac{1}{2}\lambda_k e^{2\pi iw }\right) f_{w,k}(x)
		\end{aligned}
	\end{equation*}
\end{proof}
	
	\begin{lemma}\label{eigen_F}
		Suppose $(\mathbb{Z}^n,\hat{\sigma})$ is a connection graph where $\hat{\sigma}=\sigma_1\otimes\cdots\otimes\sigma_n$ and $\{\sigma_j:E(\mathbb{Z})\rightarrow O(d_j)\}_{j=1}^n$ are constant connections. Let $(\lambda_{k_j}^{(j)},v_{k_j}^{(j)})_{k_j=1}^{d_j}$  be the orthonormal eigensystem of $\sigma_j$.
		
		Then $\forall w\in \mathbb{R}^n,1\le k_j\le d_j$, $F_w^{k_1,k_2,\cdots k_n}:=exp(2\pi i\langle w,x\rangle)v^{(1)}_{k_1}\otimes v^{(2)}_{k_2}\otimes\cdots \otimes v^{(n)}_{k_n}$ is an eigenfunction of $ \mathcal{L}^{\hat{\sigma}}$ on $(\mathbb{Z}^n,\hat{\sigma})$ .
	\end{lemma}
	
	\begin{proof}
		$\forall w_j \in \mathbb{R}, 1\le k_j\le d_j$, let $f_{w_j,k_j}^{(j)}(x_j)=exp(2\pi iw_jx_j )v^{(j)}_{k_j}$.
		
		Then  $F_w^{k_1,k_2,\cdots k_n}=f_{w_1,k_1}^{(1)}\otimes \cdots \otimes f_{w_n,k_n}^{(n)}$. From the above lemma, we know $f_{w_j,k_j}^{(j)} $ is an eigenfunction of $\mathcal{L}^{\sigma_j}$. Since $\mathcal{L}^{\hat{\sigma}}=\frac{1}{n}\mathcal{L}^{\sigma_1}\oplus \cdots \oplus \frac{1}{n}\mathcal{L}^{\sigma_n}$, we have
		\begin{equation*}
			\begin{aligned}
				&\mathcal{L}^{\hat{\sigma}}f_{w_1,k_1}^{(1)}\otimes \cdots \otimes f_{w_n,k_n}^{(n)}\\
				&=\sum_{j=1}^{n}f_{w_1,k_1}^{(1)}\otimes\cdots \otimes \left( \frac{1}{n}\mathcal{L}^{\sigma_j}f_{w_j,k_j}^{(j)}\right)\otimes\cdots \otimes f_{w_n,k_n}^{(n)}\\
				&=\frac{1}{n}\sum_{j=1}^{n}\left(1-\frac{1}{2}(\lambda_{k_j}^{(j)})^{-1} exp(-2\pi iw_j)-\frac{1}{2}\lambda^{(j)}_{k_j} exp(2\pi iw_j)\right)f_{w_1,k_1}^{(1)}\otimes \cdots \otimes f_{w_n,k_n}^{(n)}\\
				&=\frac{1}{n}\sum_{j=1}^{n}\left(1-\frac{1}{2}(\lambda_{k_j}^{(j)})^{-1} exp(-2\pi iw_j)-\frac{1}{2}\lambda^{(j)}_{k_j} exp(2\pi iw_j)\right)F_w^{k_1,k_2,\cdots k_n}
			\end{aligned}
		\end{equation*}
	\end{proof}
	
	Obviously, the constant product connection $\hat{\sigma}$ is $M\mathbb{Z}^n-$proper. Therefore, the connection discrete torus $(\mathbb{Z}^n/M\mathbb{Z}^n,\hat{\sigma}^{Q_{M\mathbb{Z}^n}})$ is well defined. The next lemma shows the eigensystem of its connection Laplacian $\mathcal{L}^{\hat{\sigma}^{Q_{M\mathbb{Z}^n}}}$. 
	
	\begin{lemma}
		The orthonormal system of $\mathcal{L}^{\hat{\sigma}^{Q_{M\mathbb{Z}^n}}}$ is
		$$
		\begin{aligned}
			&\left\{\left(\frac{1}{n}\sum_{j=1}^{n}\left(1-\frac{1}{2}\overline{\lambda_{k_j}^{(j)}} exp(-2\pi iw_j)-\frac{1}{2}\lambda_{k_j}^{(j)} exp(2\pi iw_j)\right),\frac{1}{\sqrt{det M}}F_w^{k_1,\cdots,k_n}\right)\right\} \\
			& where\quad w\in(M^*)^{-1}\mathbb{Z}^n /\mathbb{Z}^n,1\le k_1\le d_1,\cdots,1\le k_n\le d_n
		\end{aligned}
		$$
	\end{lemma}
	
	\begin{proof}
		Due to \cite[Lemma3.1]{grigor2022discrete}, 
		$F_w^{k_1,k_2,\cdots k_n}$ is $M\mathbb{Z}^n-$periodic if and only if $ w\in (M^*)^{-1}\mathbb{Z}^n /\mathbb{Z}^n$.
		From lemma \ref{solutionlemma} and lemma \ref{eigen_F}, we know $F_w^{k_1,k_2,\cdots k_n}$ is an eigenfunction of $\mathcal{L}^{\hat{\sigma}^{Q_{M\mathbb{Z}^n}}}$ if and only if $ w\in (M^*)^{-1}\mathbb{Z}^n /\mathbb{Z}^n$. 
		
		If $(k_1,\cdots,k_n)\neq(k_1^{'},\cdots,k_n^{'})$, it's obvious that $F_w^{k_1,\cdots,k_n}\perp F_w^{k_1^{'},\cdots,k_n^{'}} $ because $v^{(i)}_{k_i}\perp v^{(i)}_{k_i^{'}} $ when $k_i\neq k_i^{'}$. If $w,w^{'}\in (M^*)^{-1}\mathbb{Z}^n /\mathbb{Z}^n $ and $w\neq w^{'}$, then $e^{2\pi iwx}, e^{2\pi iw^{'}x}$ are eigenfunctions of Laplacian $\mathcal{L}$ on underlying graph $M\mathbb{Z}^n$ and $e^{2\pi iwx}\perp e^{2\pi iw^{'}x} $ for the reason that $0$ is a simple eigenvalue of $\mathcal{L}$ implying that $ e^{2\pi i(w-w^{'})x}$ is orthogonal to $0$'s constant eigenfunction $1$. 
		Combining with $||F_w^{k_1,\cdots,k_n}||_{l^2}=\sqrt{det M}$, we finish the proof.  
	\end{proof}
	
	Let $H_t^{\hat{\sigma}^{Q_{M\mathbb{Z}^n}}}$ be connection heat kernel on $(\mathbb{Z}^n/M\mathbb{Z}^n,\hat{\sigma}^{Q_{M\mathbb{Z}^n}})$. Then we can derive an alternative expression of $H_t^{\hat{\sigma}^{Q_{M\mathbb{Z}^n}}}$.
	
	\begin{theorem}\label{torus 2}
		\begin{equation*}
			\begin{aligned}
				&H_t^{\hat{\sigma}^{Q_{M\mathbb{Z}^n}}}\left(\left(x_1,\cdots,x_n\right),\left(y_1,\cdots,y_n\right)\right)\\
				&=\frac{1}{det M}\sum_{w\in(M^*)^{-1}\mathbb{Z}^n /\mathbb{Z}^n}e^{-t}e^{2\pi i\langle w,x-y\rangle}e^{\frac{t}{n}cos(2\pi w_1)\sigma_1\oplus\frac{t}{n}cos(2\pi w_2)\sigma_2\oplus\cdots\oplus\frac{t}{n}cos(2\pi w_n)\sigma_n }
			\end{aligned}
		\end{equation*}
	\end{theorem}
	\begin{proof}
		Putting the orthonormal system into the characteristic representation of connection  heat kernel, we have
		
		\begin{align*}
			&H_t^{\hat{\sigma}^{Q_{M\mathbb{Z}^n}}}\left(\left(x_1,\cdots,x_n\right),\left(y_1,\cdots,y_n\right)\right)\\
			&=\frac{1}{det M}\sum_{w\in(M^*)^{-1}\mathbb{Z}^n /\mathbb{Z}^n}\sum_{1\le k_1\le d_1,\cdots,1\le k_n\le d_j }\\
			&\exp\left(\frac{-t}{n}\sum_{j=1}^{n}\left(1-\frac{1}{2}\overline{\lambda_{k_j}^{(j)}} exp(-2\pi iw_j)-\frac{1}{2}\lambda_{k_j}^{(j)} exp(2\pi iw_j)\right)\right)\\ &\left(f_{w_1,k_1}^{(1)}(x_1)\overline{\left(f_{w_1,k_1}^{(1)}(y_1)\right)}^T\right)\otimes \cdots \otimes \left(f_{w_n,k_n}^{(n)}(x_n)\overline{\left(f_{w_n,k_n}^{(n)}(y_n)\right)}^T\right)\\
			&=\frac{1}{det M}\sum_{w\in(M^*)^{-1}\mathbb{Z}^n /\mathbb{Z}^n}e^{-t}\\
			&\left(\sum_{1\le k_1\le d_1}e^{\frac{t}{n}\frac{exp(2\pi i w_1)\lambda_{k_1}^{(1)}+\overline{exp(2\pi i w_1)\lambda_{k_1}^{(1)}}}{2}}f_{w_1,k_1}^{(1)}(x_1)\overline{\left(f_{w_1,k_1}^{(1)}(y_1)\right)}^T\right)\\
			&\otimes\cdots\otimes\left(\sum_{1\le k_n\le d_n}e^{\frac{t}{n}\frac{exp(2\pi i w_n)\lambda_{k_n}^{(n)}+\overline{exp(2\pi i w_n)\lambda_{k_n}^{(n)}}}{2}}f_{w_n,k_n}^{(n)}(x_1)\overline{\left(f_{w_n,k_n}^{(n)}(y_1)\right)}^T\right)\\
			&=\frac{1}{det M}\sum_{w\in(M^*)^{-1}\mathbb{Z}^n /\mathbb{Z}^n}e^{-t}\\
			&\left(\sum_{1\le k_1\le d_1}e^{2\pi i w_1(x_1-y_1)}e^{\frac{t}{n}\frac{exp(2\pi i w_1)\lambda_{k_1}^{(1)}+\overline{exp(2\pi i w_1)\lambda_{k_1}^{(1)}}}{2}}v_{k_1}^{(1)}\overline{\left(v_{k_1}^{(1)}\right)}^T\right)\\
			&\otimes\cdots\otimes\left(\sum_{1\le k_n\le d_n}e^{2\pi i w_n(x_n-y_n)}e^{\frac{t}{n}\frac{exp(2\pi i w_n)\lambda_{k_n}^{(n)}+\overline{exp(2\pi i w_n)\lambda_{k_n}^{(n)}}}{2}}v_{k_n}^{(n)}\overline{\left(v_{k_n}^{(n)}\right)}^T\right)\\
		\end{align*}
		
		For $\forall j=1,2,\cdots,n$, since $\left\{ (\lambda_{k_j}^{(j)},v_{k_j}^{(j)})\right\}_{k_j=1}^{d_j}$ is the orthonormal system of $\sigma_j$, 
		
		 $\left\{ (\frac{exp(2\pi i w_j)\lambda_{k_j}^{(j)}+\overline{exp(2\pi i w_j)\lambda_{k_j}^{(j)}}}{2},v_{k_j}^{(j)})\right\}_{k_j=1}^{d_j}$ is the orthonormal system of $cos(2\pi w_j)\sigma_j$.
		Therefore, for $\forall j=1,2,\cdots,n$
		\[ e^{\frac{t}{n}cos(2\pi w_j)\sigma_j}=\sum_{1\le k_j\le d_j}e^{\frac{t}{n}\frac{exp(2\pi i w_j)\lambda_{k_j}^{(j)}+\overline{exp(2\pi i w_j)\lambda_{k_j}^{(j)}}}{2}}v_{k_j}^{(j)}\overline{v_{k_j}^{(j)}}^T \]
		Combining the above two equations, we have 
		\begin{align*}
			&H_t^{\sigma^{\mathbb{Z}^n/M\mathbb{Z}^n}}\left(\left(x_1,\cdots,x_n\right),\left(y_1,\cdots,y_n\right)\right)\\
			&=\frac{1}{det M}\sum_{w\in(M^*)^{-1}\mathbb{Z}^n /\mathbb{Z}^n}e^{-t}e^{2\pi i w_1(x_1-y_1)}e^{2\pi i w_2(x_2-y_2)}\cdots e^{2\pi i w_n(x_n-y_n)}\\
			&e^{\frac{t}{n}cos(2\pi w_1)\sigma_1}\otimes\cdots\otimes e^{\frac{t}{n}cos(2\pi w_n)\sigma_n}\\
			&=\frac{1}{det M}\sum_{w\in(M^*)^{-1}\mathbb{Z}^n /\mathbb{Z}^n}e^{-t}e^{2\pi i\langle w,x-y\rangle}e^{\frac{t}{n}cos(2\pi w_1)\sigma_1\oplus\frac{t}{n}cos(2\pi w_2)\sigma_2\oplus\cdots\oplus\frac{t}{n}cos(2\pi w_n)\sigma_n }
		\end{align*}	 
	\end{proof}

	As a result of theorem \ref{torus 1} and the theorem \ref{torus 2}, we get the following equality:
	
	\begin{theorem}\label{mainthm}
		Suppose $M$ is an integer $n\times n$ matrix and $detM>1$ and $\sigma_1,\cdots,\sigma_n$ are arbitrary orthogonal matrices, then
		\begin{tiny}
			\begin{equation*}
				\begin{aligned}
					&\sum_{a\in M\mathbb{Z}^n}\prod_{i=1}^{n}\left((-1)^{|y_i+a_i-x_i|} \sum_{k\ge 0}\frac{C_{2k}^{k+|y_i+a_i-x_i|}}{k!} (-\frac{t}{2n})^k\right) \sigma^{y_1+a_1-x_1}_1\otimes\cdots\otimes \sigma^{y_n+a_n-x_n}_n\\
					&=\frac{1}{det M}\sum_{w\in(M^*)^{-1}\mathbb{Z}^n /\mathbb{Z}^n}e^{2\pi i\langle w,x-y\rangle}e^{-t}e^{\frac{t}{n}cos(2\pi w_1)\sigma_1\oplus\frac{t}{n}cos(2\pi w_2)\sigma_2\oplus\cdots\oplus\frac{t}{n}cos(2\pi w_n)\sigma_n }
				\end{aligned}
			\end{equation*} 
		\end{tiny}
	\end{theorem} 
	
	\section{The Application of Connection Heat Kernel }
	
	\subsection{Connection Trace Formula on Connection Discrete Torus}
	When $x=y$, Theorem \ref{mainthm} has the following form:
	\begin{equation}\label{trace}
		\begin{aligned}
			&\sum_{a\in M\mathbb{Z}^n}\prod_{i=1}^{n}\left((-1)^{|a_i|} \sum_{k\ge 0}\frac{C_{2k}^{k+|a_i|}}{k!} (-\frac{t}{2n})^k\right) \sigma^{a_1}_1\otimes\cdots\otimes \sigma^{a_n}_n\\
			&=\frac{1}{det M}\sum_{w\in(M^*)^{-1}\mathbb{Z}^n /\mathbb{Z}^n}e^{-t}e^{\frac{t}{n}cos(2\pi w_1)\sigma_1\oplus\frac{t}{n}cos(2\pi w_2)\sigma_2\oplus\cdots\oplus\frac{t}{n}cos(2\pi w_n)\sigma_n }
		\end{aligned}
	\end{equation} 
	In \cite[Thm 9]{chung1997combinatorial} cycle's trace formula is as follows:
	\[ \sum_{j\in \mathbb{Z}}\sum_{k\ge 0}\left((-1)^{mj} \frac{C_{2k}^{k+mj}}{k!} (-\frac{t}{2})^k\right) =\frac{1}{m}\sum_{w=0}^{m-1} exp(-t(1-cos(\frac{2\pi w}{m})))\]
	
	We call the equation (\ref{trace}) connection trace formula on connection discrete torus, which can be regarded as a promotion for trace formula on cycle. 
	
	\textbf{Modified Bessel functions} For integer $x$ and parameter $t\ge 0$, the modified Bessel function is \[ I_x(t)=\frac{1}{\pi}\int_{0}^{\pi}e^{tcos\theta}cos(x\theta)d\theta=\sum_{k=0}^{\infty}\frac{(\frac{t}{2})^{x+2k}}{k!(x+k)!}  \]
	From the property $I_{x-1}(t)+I_{x+1}(t)=2\frac{\partial}{\partial t}I_x(t)$ and $I_x(0)=0$ unless $x=0$, it's easy to check $e^{-t}I_x(t)$ is the solution of heat equation $(\frac{\partial}{\partial t}+\mathcal{L})f(t,x)=0$ on $\mathbb{Z}$ with initial condition $f(0,x)=\delta_0(x)$, suggesting that $e^{-t}I_{y-x}(t)$ is the solution under initial condition $f(0,x)=\delta_y(x)$ . Since $H^{\mathbb{Z}}(x,y)$ is the solution of heat equation under initial condition $f(0,x)=\delta_y(x)$, $H^{\mathbb{Z}}(x,y)=e^{-t}I_{y-x}(t)$. In \cite[Thm4]{chung1997combinatorial}, we know
	$$H^{\mathbb{Z}}(x,y)=(-1)^{|y-x|} \sum_{k\ge 0}\frac{C_{2k}^{k+|y-x|}}{k!} (-\frac{t}{2})^k $$ Therefore, 
	\begin{align*}
		(-1)^{|x|} \sum_{k\ge 0}\frac{C_{2k}^{k+|x|}}{k!}(-\frac{t}{2})^k&=e^{-t}I_{x}(t)\\
		\prod_{i=1}^{n}\left((-1)^{|a_i|} \sum_{k\ge 0}\frac{C_{2k}^{k+|a_i|}}{k!} (-\frac{t}{2n})^k\right)&=\prod_{i=1}^{n}e^{-\frac{t}{n}}I_{a_i}(\frac{t}{n})
	\end{align*}
	Putting the above equation into equation (\ref{trace}), we get
	\begin{align*}
		&\sum_{a\in M\mathbb{Z}^n}\prod_{i=1}^{n}e^{-\frac{t}{n}}I_{a_i}(\frac{t}{n}) \sigma^{a_1}_1\otimes\cdots\otimes \sigma^{a_n}_n\\
		&=\frac{1}{det M}\sum_{w\in(M^*)^{-1}\mathbb{Z}^n /\mathbb{Z}^n}e^{-t}e^{\frac{t}{n}cos(2\pi w_1)\sigma_1\oplus\frac{t}{n}cos(2\pi w_2)\sigma_2\oplus\cdots\oplus\frac{t}{n}cos(2\pi w_n)\sigma_n }
	\end{align*}
	We call the above equality as the connection theta relation, which can be regarded as a promotion for theta relation in \cite[Thm 1]{karlsson2006heat}.
	
	\subsection{Vector Diffusion Distance Based on Connection Heat Kernel}
	
	In processing and analyzing an immense amount of high dimensional data sets, weighted graphs are often used to represent the affinities between data points. Consequently, many dimensionality reduction methods have appeared in the past decade, such as diffusion map\cite{coifman2006diffusion}, locally linear embedding\cite{roweis2000nonlinear} and so on. A new mathematical framework called vector diffusion map(VDM)\cite{singer2012vector} utilizes connection graph to symbolize high data, which assigns every edge of graph not only a weight but also a linear orthogonal transform. Based on connection kernel, VDM defines an embedding of data into a Hilbert space and the distance between data points is called vector diffusion distance. In \cite[Thm 8.2]{singer2012vector}, we find the vector diffusion distance behaves like geodesic distance in asymptotic limit. 
	
	In manifold setup, let $\mathcal{M}$ be a manifold and $\triangledown^2$ is connection Laplacian on $\mathcal{M}$. Assume $\left\{(-\lambda_k,X_k)\right\}_{k=0}^{\infty}$ is the orthonormal eigensystem of $\triangledown^2$, the connection heat kernel \cite{grigor2001heat} is \[ H_t(x,y)=\sum_{n=0}^{\infty} e^{-t\lambda_n}X_n(x)\otimes \overline{X_n(y)} \]
	Define diffusion map $V_t:\mathcal{M}\rightarrow \mathcal{l}^2$ as following:
	\[ V_t:x\rightarrow \left(e^{-\frac{t(\lambda_m+\lambda_n)}{2}} \langle X_n(x), X_m(x)\rangle \right)_{n,m=0}^{\infty}  \]
	$\forall x,y\in \mathcal{M}$, the vector diffusion distance is defined as \[ d_{VDM,t}(x,y):=||V_t(x)-V_t(y)||_{\mathcal{l}^2} \]
	We can see how connection heat kernel is pertinent to vector diffusion distance from the following direct calculation:
	\begin{equation*}
		\begin{aligned}
			||H_t(x,y)||_{HS}^2&=tr[H(x,y)H(x,y)^*]\\
			&=\sum_{n,m=0}^{\infty}e^{-t(\lambda_n+\lambda_m)} \langle X_n(x), X_m(x)\rangle \overline{\langle X_n(x), X_m(x)\rangle}\\
			&=\langle V_t(x),V_t(y) \rangle_{\mathcal{l}^2}
		\end{aligned}
	\end{equation*} 
	\begin{equation*}
		\begin{aligned}
			d_{VDM,t}^2(x,y)&=\langle V_t(x)-V_t(y),V_t(x)-V_t(y)\rangle_{\mathcal{l}^2} \\
			&=\langle V_t(x),V_t(x)\rangle_{\mathcal{l}^2}+\langle V_t(y),V_t(y)\rangle_{\mathcal{l}^2}-2\langle V_t(x),V_t(y)\rangle_{\mathcal{l}^2}\\
			&=||H_t(x,x)||_{HS}^2+||H_t(y,y)||_{HS}^2-2||H_t(x,y)||_{HS}^2
		\end{aligned}
	\end{equation*}
	\bibliographystyle{plain}
	\bibliography{ref2023}
	
\end{document}